\documentclass[12pt, a4paper]{article}
\usepackage[english]{babel}
\usepackage[iso-8859-7]{inputenc}
\usepackage{latexsym}
\usepackage{amsmath, amsthm, amssymb}
\usepackage{amssymb}
\usepackage{amsmath}
\usepackage{amsfonts}
\usepackage{amsthm}
\usepackage[all]{xy}
\usepackage{indentfirst}
\input{epsf.tex}

\newtheorem{definition}{Definition}[section]
\newtheorem{theorem}{Theorem}[section]

\newtheorem{corollary}{Corollary}[section]

\newcommand{\nn}{\mathbb{N}}
\newcommand{\cc}{\mathbb{C}}

\newcommand{\di}{\displaystyle}

\newcommand{\st}{\subset}

\begin{document}
\title{\bf Subclasses of universal Taylor series and center independence}

\author{V.Vlachou}         
\date{}
\maketitle 
\begin{abstract} 
A holomorphic function on a simply connected domain $\Omega$ belongs to a subclass of universal Taylor series if prescribed and  infinite number of  partial sums of the Taylor expansion of $f$ around a given center $\zeta_0$ realize Mergalyan-type approximations outside $\Omega$. We will prove that this class is independent of the choice of center if the indeces of the partial sums do not grow very fast to infinity.
\end{abstract}
\noindent
\section{Introduction }
In the last 30 years, several authors have   worked on the notion of universality and  important advances in the research of this topic
have been made. Roughly speaking, we say that an object is universal, if a family (usually denumerable) of objects produced by the original (by a pattern), diverge in a maximal way or in other words realize maximal approximations.

We are interested in Universal Taylor Series, a notion first introduced by V. Nestoridis (see \cite{N2}). For this reason, let us fix a simply connected domain $\Omega\st\cc$.  We will start by giving some basic notation.   We denote by $H(\Omega)$  the space
of functions, holomorphic in $\Omega$, endowed with the topology of uniform convergence on compacta. 
Moreover, for $n\in\nn$, $\zeta_0\in\Omega$ and $f\in H(\Omega)$  we denote by 
$S_{n}(f,\zeta_0)$ the $n^{th}$ partial sum of the Taylor expansion of $f$ around $\zeta_0$.\\
Additionally,  we also use the notations 
\begin{itemize}
\item $\mathcal{A}(K)=\{g\in C(K): \ g \text{ is
holomorphic in } K^o\}, \  K\subset\cc \text{ compact }$\\ 
\item $\di||g||_K=\sup_{z\in K} |g(z)|, \ g\in \mathcal{A}(K)$\\
 \item $\mathcal{M}=\{K\st
\cc: K \ \text{is a compact set with } \ K^c \text{ connected
set}\}$\\
\item $\mathcal{M}_{\Omega^c}=\{K\in\mathcal{M}: K\cap \Omega=\emptyset\}$.
\end{itemize}
We are now ready to give the definition of Universal Taylor Series as given in \cite{N3} (see also \cite{N2}):
\begin{definition}\label{Ne} A holomorhic function $f:\Omega\to\cc$ belongs to the class $U(\Omega, \zeta_0)$, for a fixed point $\zeta_0\in\Omega$, if the set
$\{S_{n}(f,\zeta_0): \ n\in\nn\}$ is dense in $\mathcal{A}(K)$ for every $K\in\mathcal{M}_{\Omega^c}$.
\end{definition} 
In \cite{N3} V. Nestoridis proved that the class $U(\Omega, \zeta_0)$ is a $G_{\delta}$ and dense subset of $H(\Omega)$ for every $\zeta_0\in\Omega$. Moreover,  the class of Universal Taylor Series $U(\Omega,\zeta_0)$   with respect to a fixed center $\zeta_0\in\Omega$   is independent of the choice of $\zeta_0$. This is a deep and elegant result
that first appeared in \cite{GLM} for $\Omega$ bounded and in \cite{vmy1} for the general case (see also \cite{N1}). The motivation  of 
this article lies on this result. In order to be more specific, we start by  giving  the definition of a subclass of $U(\Omega, \zeta_0)$ which has first appeared in the  bibliography in order to prove algebraic genericity of the class $U(\Omega, \zeta_0)$ (see \cite{BGNP}).
\begin{definition}\label{sub} A holomorhic function $f:\Omega\to\cc$ belongs to the class $U_{\lambda }(\Omega, \zeta_0)$, for a fixed  sequence of positive integers $\lambda\equiv(\lambda_n)_{n}$ and a fixed point $\zeta_0\in \Omega$, if the set
$\{S_{\lambda_n}(f,\zeta_0), \ n\in\nn\}$ is dense in $\mathcal{A}(K)$ for every $K\in\mathcal{M}_{\Omega^c}$.
\end{definition}  
If $\Omega$ is a simply connected domain, then  the class   $U_{\lambda }(\Omega, \zeta_0)$  is non-empty, if and only if,  the sequence $(\lambda_n)_{n}$ is unbounded  (see \cite{BGNP}).\\ \textbf{Question: }Is this class of functions independent of the choice of $\zeta_0$ or not?\\
This question was  stated in \cite{vla}, where it was proved that if $\lambda=(\lambda_n)_n=n^{\sigma}$ for some $\sigma\in\nn$, then the answer is positive.

Our goal is to study  center independence for more general sequences.
 Our main result implies  that if  $\lambda=(\lambda_n)_n=([P(n)])_n$, where $P$ is a polynomial with $\di\lim_n P(n)=+\infty$  or if $\lambda=(\lambda_n)_n=(a^n)_n$, $a>1$, then the answer to the above question is again affirmative.  We use methods and ideas presented  in \cite{vla}.

\section{Independence of center  of expansion}

Let us  start by giving the definition of Ostrowski-gaps, since they will play a central role in this section.
\begin{definition} Let $\di\sum_{k=0}^{\infty}a_{k} (z-\zeta_0)^k$ be a power series with positive radious of convergence. We say that it has Ostrowski gaps $(p_m,q_m), \ m=1,2,\ldots$, if there
exist two sequences of positive integers $(p_m)_{m\in\nn}$ and $(q_m)_{m\in\nn}$ such that the following hold:\\
(i) $p_1<q_1\leq p_2<q_2\leq\ldots$ and $\di\lim_{m}\frac{q_m}{p_m}=+\infty$\\
(ii)For $\di I=\bigcup_{m=1}^{\infty}\{p_m+1,\ldots,q_m\}$, we have $\di\lim_{\nu\in I}|a_{\nu}|^{\frac{1}{\nu}}=0.$
\end{definition}
Let $(\lambda_{n})_n$ be an unbounded sequence of positive integers such that for every strictly increasing sequence of positive integers $(n_k)_k$, these exist two sequences of positive integers $(p_k)_k,(q_k)_k$ with:\\
1) $(q_k)_k$ is a subsequence of $(n_k)_k$.\\
2) $(p_k)_k$ is striclty increasing.\\
3) $\lambda_{p_k}<\lambda_{q_k}\leq \lambda_{p_{k+1}}, \forall k\in\nn$.\\
4)$\dfrac{\lambda_{q_k}}{\lambda_{p_k}}\to +\infty$ and $\dfrac{\lambda_{q_k}}{\lambda_{p_k}}\leq  k, \ \forall k\in\nn$. 
\begin{theorem}\label{inde} If $\lambda=(\lambda_n)_n$ is as described above. Then the class $U_{\lambda }(\Omega, \zeta_0)$ is independent of the choice of center  $\zeta_0$. 
\end{theorem}
\begin{proof}Let $f\in U_{\lambda }(\Omega, \zeta_0)$, 
 $K\in\mathcal{M}_{\Omega^c}$ and $g\in \mathcal{A}(K)$.
In view of Lemma 2.2 in \cite{vla}( see also Lemma 1 in \cite{vmy1}), there exists  an increasing  sequence of compact sets $(E_k)_{k\in\nn}$ that satisfies the following:\\
     (i)  $E_k\in\mathcal{M}_{\Omega^c}, \ k=1,2,\ldots$\\
   (ii) $\bigcup_{k}E_k$ is closed and non-thin at $\infty$,\\
   (iii) Every  $E_k$ is disjoint from $K$.\\
   
Since $f\in U_{\lambda }(\Omega, \zeta_0)$,  there exists a strictly increasing sequence $(n_k)_{k}$ such that:
$$ ||S_{\lambda_{n_k}}(f, \zeta_0)-g||_{K}\xrightarrow{k\to\infty}0,   $$          
 $$  ||S_{\lambda_{n_k}}(f,\zeta_0)||_{E_k}\xrightarrow{k\to\infty}0.    $$     \
The sequence  $( S_{\lambda_{n_k}}(f,\zeta_0))_k$ is pointwise convergent on $ E=\bigcup_k E_k$. Thus, 
for every $z\in  E=\bigcup_k E_k$,
$$ \limsup_{k}|S_{\lambda_{n_k}}(f,\zeta_0)(z)|^{\frac{1}{\lambda_{n_k}}}\leq 1.$$
 Note that for every $k$, the function $S_{\lambda_{n_k}}(f,\zeta_0)$ is a  polynomial of degree less or equal to $\lambda_{n_k}$.   Therefore in view of Lemma 2 in \cite{M-Y}  we conclude that:
$$ \limsup_{k}||S_{\lambda_{n_k}}(f,\zeta_0)||^{\frac{1}{\lambda_{n_k}}}_{|z|=R}\leq 1, \ \ \forall R>0.$$
Passing to a subsequence, we may assume that:
$$||S_{\lambda_{n_k}}(f,\zeta_0)||^{\frac{1}{\lambda_{n_k}}}_{|z|=3^k}\leq 2,  \   k=1,2,\ldots.$$
Now according to our assumptions, we may find two striclty increasing sequences of positive interges $(p_k)_k, (q_k)_k$ that satisfy 1-4.\\
Note that:
$$||S_{\lambda_{q_k}}(f,\zeta_0)||^{\frac{1}{\lambda_{q_k}}}_{|z|=3^k}\leq 2,  \   k=1,2,\ldots.$$
Assume that
$\di \sum_{\nu=0}^{\infty}a_{\nu} (z-\zeta_{0})^{\nu}$ is the Taylor expansion of $f$ around $\zeta_{0}$.
Then using Cauchy estimates  for $\lambda_{p_k}\leq\nu\leq  \lambda_{q_k}$ we obtain:
$$|a_{\nu}|^{\frac{1}{\nu}}\leq \frac{||S_{\lambda_{q_k}}(f, \zeta_0)||^{\frac{1}{\nu}}_{|z|=3^k}}{3^k}=\frac{\biggl(||S_{\lambda_{q_k}}(f, \zeta_0)||^{\frac{1}{\lambda_{q_k}}}_{|z|=3^k}\biggl)^{\frac{\lambda_{q_k}}{\nu}}}{3^k} \leq
\frac{2^{\frac{\lambda_{q_k}}{\lambda_{p_k}}}}{3^k}\leq \frac{2^{ k}}{3^k}\to 0.$$
Thus,    the  corresponding power series $\di \sum_{k=0}^{\infty}a_{\nu} (z-\zeta_{\sigma})^{\nu}$  has Ostrowski-gaps $(\lambda_{p_k}, \lambda_{q_k})$, $ k=1,2,\ldots$.\\
It is known (see \cite{Luh}) that in this case:
$$S_{\lambda_{q_k}}(f, \zeta_0)- S_{\lambda_{p_k}}(f,\zeta_0)\xrightarrow{k\to \infty} 0, \text{ compactly on } \cc.$$
Moreover, in view of Lemma 9.2 in  \cite{N1} (see also Theorem 1 in  \cite{Luh}) we have:
$$\sup_{\zeta\in L}\sup_{z\in K}|S_{\lambda_{p_k}}(f, \zeta)(z)-S_{\lambda_{p_k}}(f,\zeta_0)(z)|\xrightarrow{k\to \infty} 0$$
for every choice of compact set $L\st\Omega$.\\
Thus:
$$\sup_{\zeta\in L} ||S_{\lambda_{p_k}}(f, \zeta)-g||_{K}\xrightarrow{k\to\infty}0,$$
  for every $ L\st\Omega$  compact.\\
Therefore the center independence follows. 
\end{proof}
Actually, the proof above says more than what is stated in Theorem \ref{inde}. Let us give one more definition to describe what we mean:
\begin{definition} A holomorhic function $f:\Omega\to\cc$ belongs to the class $U_{\lambda }(\Omega)$ for a fixed  sequence of positive integers $\lambda\equiv(\lambda_n)_{n}$,  if  for every compact set $K\in \mathcal{M}_{\Omega^c}$ and every function $g\in \mathcal{A}(K) $, there exists a subsequence $(\lambda_{\mu_n})_n$ of  $(\lambda_n)_{n}$ such that:
$$\sup_{\zeta\in \Gamma}||S_{\lambda_{\mu_n}}(f,\zeta)-g||_{K}\xrightarrow{n\to+\infty}0,$$
for every compact set $\Gamma\st\Omega$.
\end{definition}
\begin{corollary}\label{corinde}If $\Omega$ is simply connected and $\lambda$ is as in Theorem \ref{inde}, then the classes $U_{\lambda}(\Omega,\zeta_0)$ and $U_{\lambda}(\Omega)$ coincide. 
\end{corollary}
\begin{proof} This follows directly from the proof of Theorem \ref{inde}.
\end{proof}
Let us, now, present two applications of Corollary\ref{corinde}.
\begin{theorem}
Let $P$ be a polynomial such that $P(n)\xrightarrow{n\to+\infty} +\infty$. If \\ $\lambda=(\lambda_n)_n=([P(n)])_n$, then $U_{\lambda}(\Omega, \zeta_0)=U_{\lambda}(\Omega),$ for all $\zeta_0\in\Omega$.
\end{theorem}
\begin{proof} We will prove that the sequence $\lambda$ satisfies conditions (1)-(4) and the result will follow in view of Corollary \ref{corinde}.
Fix a strictly increasing sequence of positive integers $(n_k)_k$. We choose a subsequence $(q_k)_k$ of  
$(n_k)_k$ such that $P(q_k)>\sqrt{ k} P(q_{k-1}), \ \forall k\in\nn$.\\
For $k$ large enough we set:
$$p_k=[P^{-1}\biggl(\dfrac{P(q_k)}{\sqrt{ k}}\biggl)]+1,$$
where $P^{-1}$ is the inverse of $P$ on an interval of the form $[a,+\infty)$ (note that $P$ is  striclty increasing function on an interval of this form).\\
Then:
$$p_k-1\leq P^{-1}\biggl(\dfrac{P(q_k)}{\sqrt{ k}}\biggl) <p_k\Rightarrow P(p_k-1)\leq \dfrac{P(q_k)}{\sqrt{ k}}<P(p_k) \Rightarrow$$
$$\Rightarrow \dfrac{P(p_k-1)}{P(p_k)}\sqrt{ k}\leq \dfrac{P(q_k)}{P(p_k)}<\sqrt{ k}\Rightarrow$$
$$\Rightarrow \dfrac{P(p_k-1)}{P(p_k)}\sqrt{ k}-\dfrac{1}{P(p_k)}\leq \dfrac{[P(q_k)]}{[P(p_k)]}<\sqrt{ k}\dfrac{P(p_k)}{P(p_k)-1}.$$
Thus all conditions are satisfied and the result follows.
\end{proof}
\begin{theorem}
Let $\lambda=(\lambda_n)_n$ be a sequence of positive integers such that 
$$\theta<\dfrac{\lambda_{n+1}}{\lambda_n}<M, \ \forall n\in\nn,$$
with $1<\theta<M$.
Then $U_{\lambda}(\Omega, \zeta_0)=U_{\lambda}(\Omega),$ for all $\zeta_0\in\Omega$.
\end{theorem}
\begin{proof} Let $(n_k)_k$ be a striclty increasing sequence of positive integers. We choose $(q_k)_k$ a subsequence of $(n_k)_k$ with 
$$q_{k+1}>q_k+\dfrac{\log (k+1)}{\log M},$$
(without loss of generality we may assume that $\log M>0$).\\
We set:
$$p_k=q_k-\biggl[\dfrac{ \log k}{\log M}\biggl]\Rightarrow q_k-p_k=\biggl[\dfrac{ \log k}{\log M}\biggl].$$
Then:
$$q_k>p_k$$
$$p_{k+1}=q_{k+1}-\biggl[\dfrac{ \log (k+1)}{\log M}\biggl]>q_k.$$
Since $\lambda$ is striclty increasing $\lambda_{p_k}<\lambda_{q_k}<\lambda_{p_{k+1}}$.
Moreover:
$$\theta^{q_k-p_k}<\dfrac{\lambda_{q_k}}{\lambda_{p_k}}<M^{q_k-p_k}.$$
and the result follows.
\end{proof}
\noindent\textbf{Open Question:} Does the result hold if $\lambda_n=n!$?

V.Vlachou 
\\
Department of Mathematics,\\
University of Patras,\\
26500 Patras,GREECE\\
e-mail: vvlachou@math.upatras.gr

\end{document}